\documentclass[twoside,11pt,a4paper]{article}
\usepackage{natbib}
\usepackage{amsmath,float,amssymb,mathtools,amsthm,mathrsfs,multirow,booktabs,setspace,morefloats,epsfig,caption,graphicx,inputenc,subcaption}
\usepackage{bigints}
\usepackage{xcolor}
\usepackage{placeins}
\newtheorem{lemma}{Lemma}%[section]
\newtheorem{thm}{Theorem}%[section]
\usepackage{fancyhdr}
\newcommand\shorttitle{U-statistic based on overlapping sample spacings}

\newcommand\authors{Rahul Singh and Neeraj Misra}

\title{U-statistic based on overlapping sample spacings}

\author{Rahul Singh$^1$ and Neeraj Misra$^2$}
\date{
	\small{Department of Mathematics and Statistics, Indian Institute of Technology Kanpur, India}\\%
}

\begin{document}
	\maketitle
	{\let\thefootnote\relax\footnote{{ Email: $^1$sirahul@iitk.ac.in,  $^2$neeraj@iitk.ac.in}}}
	
	\begin{abstract}
		For testing goodness of fit, we consider a class of U-statistics of overlapping spacings of order two, and investigate their asymptotic properties. The standard U-statistic theory is not directly applicable here as the overlapping spacings form a dependent random sequence. The asymptotic distribution of the statistics under the null hypothesis and under a sequence of local alternatives are derived. In terms of the Pitman ARE, the U-statistic based on Gini's mean square difference of overlapping spacings is found to be the asymptotically locally most powerful. Interestingly, this test has the same efficacy as the Greenwood test based on overlapping spacings.
	\end{abstract}

\section{Introduction}
The Goodness of Fit (GoF) of a statistical model describes how well it fits to the given set of observations.  Measures of GoF typically summarize the discrepancy between observed sample values and expected hypothesized model. The widely studied goodness of fit (GoF) problem is to test whether a given set of observations come from a specified distribution. Let $X_1,\ldots,X_{n-1}$ be independently and identically distributed (iid) random variables having an absolutely continuous distribution $F$ on the real line. There exists abundant literature on the problem of testing simple GoF, i.e., for testing
\begin{align}\label{gof}
H_0:F=F_0\text{ against }H_A:F\neq F_0,
\end{align}
where $F_0$ is a completely specified absolutely continuous distribution function.
Using the probability integral transform, without loss of generality, we can assume that $F$ is supported on $[0,1]$ and the testing problem (\ref{gof}) is equivalent to testing:
\begin{align}
H_0:F(x)=x,\ \forall x\in[0,1]\text{ against }H_A:F(x)\neq x \text{ for some }x\in[0,1].
\end{align}
Let $X_{1:n}\! \leq \!X_{2:n}\! \leq \cdots \leq \! X_{n-1:n}$ denote the order statistics corresponding to the random sample $X_1,\ldots,X_{n-1}$, and let $X_{0:n}=0$ and $X_{n:n}=1$. For any positive integer $m$ ($\leq n/2$), the overlapping $m$-spacings are defined as:
\begin{align}
D_{j,n}^{(m)}=
X_{j+m-1:n}-X_{j-1:n},\ &j=1,2,\ldots, n-m+1.
\end{align}
Under $H_0$, we denote observations $X_i$'s by $U_i$'s and overlapping $m$-spacings by
\begin{align}
T_{j,n}^{(m)}=
U_{j+m-1:n}-U_{j-1:n},\ &j=1,2,\ldots, n-m+1.
\end{align}
Let $h:[0,\infty)\times[0,\infty)\to\mathbb{R}$ be a function that is symmetric in its arguments (i.e., $h(x,y)=h(y,x),~\forall (x,y)\in[0,\infty)\times[0,\infty$)). Consider statistics of the form
\begin{align}\label{two}
W_{m,n}(h)=\frac{2}{(n-m+1)(n-m)}\sum_{1\leq i<j\leq n-m+1}h\left(nD_{i,n}^{(m)},nD_{j,n}^{(m)}\right),
\end{align}
which is a second order U-statistic of the overlapping $m$-spacings. A popular example of such statistic is the generalized Gini's mean difference of overlapping $m$-spacings, i.e.,
\begin{align}
G_{m,n}(r)=\frac{2}{(n-m+1)(n-m)}\sum_{1\leq i<j\leq n-m+1}\left|nD_{i,n}^{(m)}-nD_{j,n}^{(m)}\right|^r,\hfill r>0.
\end{align}
For $r=1$, $G_{m,n}(1)$ is the Gini's mean difference  of overlapping $m$-spacings, and for $r=2$, $G_{m,n}(2)$ is the Gini's mean squared difference  of overlapping $m$-spacings.

The most widely studied statistics based on overlapping $m$-spacings have the form 
\begin{align}\label{one}
V_{m,n}(g)=\frac{1}{n}\sum_{i=1}^{n-m+1}g\left(nD_{i,n}^{(m)}\right),
\end{align}	
where $g:[0,\infty)\rightarrow\mathbb{R}$ is a function satisfying certain smoothness assumptions. This statistic can be thought as first order U-statistic of the overlapping $m$-spacings.

For tests based on simple spacings (i.e., $m=1$), \cite{rao_1970} found that the Greenwood test (one corresponding to $g(x)=x^2,~x\geq0$) is asymptotically optimal among tests based on (\ref{one}). \cite{pino_1979} showed that the Greenwood type test based on disjoint higher order spacings is asymptotically superior to classical Greenwood test based on simple spacings ($m=1$). Kuo and Rao (1981) established that, for any fixed spacings size $m$, the Greenwood type test is asymptotically optimal among tests based on statistics of type (\ref{one}). \cite{rao_1984} found that, for any fixed spacings size $m$, the Greenwood type test based on overlapping $m$-spacings is superior to the corresponding test based on disjoint $m$-spacings.
%%%%%%
One known weakness of tests based on symmetric sum functions of spacings is that they can not detect alternatives converging to the null distribution at a  rate faster than $n^{-1/4}$.

\cite{rao_2004} studied the statistic $G_{1,n}(1)$ for GoF purpose. They obtained exact and asymptotic distribution of the statistic $G_{1,n}(1)$ under the null hypothesis. Their simulation studies suggest that the GoF test based on $G_{1,n}(1)$ has encouraging powers. 

For the special case $m=1$, second order U-statistic(s) becomes
\begin{align}\label{simple}
W_{1,n}(h)=\frac{2}{n(n-1)}\sum_{1\leq i<j\leq n}h\left(nD_{i,n}^{(1)},nD_{j,n}^{(1)}\right).
\end{align}
Asymptotic behaviour of statistics of the type (\ref{simple}) was studied by \cite{tung_2012a}. They found that the test based on $G_{1,n}(2)$ is asymptotically optimal among tests based on such statistics. This test is equivalent to the classical Greenwood test (\cite{greenwood_1946}). \cite{tung_2012b} investigated analogues of statistics of type (\ref{two}) that are based on disjoint $m$-spacings. In this case also, they found that the asymptotically optimal test based on such statistics is the one that is based on Gini's squared mean difference, and this is equivalent to the Greenwood type test based on disjoint $m$-spacings. Now, a natural question that arises is: do similar results hold true for statistics based on overlapping $m$-spacings as well? In this article, we attempt to address this question.

Rest of the paper is organized as follows. In Section 2, we derive the asymptotic null distribution of the statistic $W_{m,n}(h)$, defined in (\ref{two}), and in Section 3 we derive asymptotic distribution of $W_{m,n}(h)$ under a sequence of local alternatives that converges to the null hypothesis at a rate of $n^{-1/4}$. In Section 4, we derive the asymptotically optimal test. Finally, in Section 5, we report some finite sample simulation study.
%%%%%%%%%%%%
\section{Asymptotic null distribution}
In this section, we will discuss asymptotic distribution of second order U-statistic of the overlapping $m$-spacings, under the null hypothesis. Under $H_0$, we have 
\begin{align}
W_{m,n}(h)=\frac{2}{(n-m+1)(n-m)}\sum_{1\leq i<j\leq n-m+1}h\left(nT_{i,n}^{(m)},nT_{j,n}^{(m)}\right).
\end{align}
Let $Z_1,Z_2,\ldots, Z_{n}$ be iid standard exponential random variables. 
Take $\zeta_m= Z_1+Z_2+\cdots+Z_m$, $\zeta_{j,m}=Z_{j}+Z_{j+1}+\cdots+Z_{j+m-1},~j=1,2,\ldots,n-m+1$ and $\bar{\zeta}_n=\dfrac{Z_1+Z_2+\cdots+Z_{n}}{n}$. 
Observe that $\zeta_{j,m}\ \&\ \zeta_{k,m}$ are distributed as $gamma(m,1)$ random variables, and
$\zeta_{j,m}\ \&\ \zeta_{k,m}$ are independent if and only if $|j-k|\geq m$. 
Then, we have the following representation (cf. \cite{rao_1980})
\begin{align}\label{con_repr}
\left(nT_{1,n}^{(m)},nT_{2,n}^{(m)},\ldots,nT_{n-m+1,n}^{(m)}\right)\stackrel{d}{=}&
\left(\frac{{\zeta}_{1,m}}{\bar{\zeta}_n},\frac{{\zeta}_{2,m}}{\bar{\zeta}_n},\ldots,\frac{{\zeta}_{n-m+1,m}}{\bar{\zeta}_n}\right)\nonumber\\
\stackrel{d}{=}&
\left({\zeta}_{1,m},{\zeta}_{2,m}\ldots,{\zeta}_{n-m+1,m}\big| \bar{\zeta}_n=1\right).
\end{align}

Here $\{\zeta_{j,m}\}_{j=1}^{n-m+1}$ is a $(m-1)$-dependent stationary sequence of random variables. Many authors have studied U-statistics based on $m$-dependent sequence of random variables (e.g. \cite{wang_1999}), and found that these statistics have similar asymptotic properties as U-statistics based on iid sequence of random variables. 

%%%%%%%%%%%%%%%%%%%%%%%%%%%%%%%%%%%%%%%%%%%%%%%%%%%%%%%%%%%%%%%%%%%%%%%%%%%%%
Define $\theta=\mathbb{E}[h(\zeta_{1,m},\zeta_{m+1,m})]$ and let 
\begin{align}
U_N=\frac{2}{N(N-1)}\sum_{1\leq i<j\leq N}h\left(\frac{{\zeta}_{i,m}}{\bar{\zeta}_n},\frac{{\zeta}_{j,m}}{\bar{\zeta}_n}\right),
\end{align}
where $N=n-m+1$. Then, $U_N$ is a second order U-statistic based on conditionally $(m-1)$-dependent sequence of random variables $\{\frac{{\zeta}_{j,m}}{\bar{\zeta}_n}\}_{j=1}^{N}$ with the kernel $h:[0,\infty)\times[0,\infty)\to\mathbb{R}$.
We require following assumptions for asymptotic results:
\begin{itemize}
\item[\textbf{A1.}] $\mathbb{E}[\left|h(\zeta_{1,m},\zeta_{j,m})\right|^3]<\infty,\ \forall j=1,2,\ldots,m+1$.
\end{itemize}
%%%%%%%%%%%%%%%%%%%%%%%%%%%%%%%%%%%%%%%%%%%%%%%%%%%%%%%%%%%%%%%%%%%%%%%%%%%%%%%%%%%%%%%%%%%
%%%%%%%%%%%%%%%%%%%%%%%%%%%%%
For $m=1$, denote $T_{j,n}^{(1)}$ by $T_{j,n}$. From \cite{holst_1979} we have the following result.
\begin{lemma}\label{lemma1}
Let $\psi:\mathbb{R}^M\to \mathbb{R}$ be any real measurable function. For $1\leq M\leq n-3$,
\begin{align*}
	&\mathbb{E}\big[e^{it\psi(nT_{1,n},nT_{2,n},\ldots,nT_{M,n})}\big]\\
	=&~\frac{(n-1)!}{2\pi n^{n-1}e^{-n}} 
	\int_{-\infty}^{\infty}\mathbb{E}\bigg[e^{it\psi(Z_1,Z_2,\ldots, Z_M)+iu\sum_{j=1}^{n}(Z_j-1)}\bigg]~du.
\end{align*}
\end{lemma}
Observe that $\{\zeta_{j,m}\}_{j\geq1}$ is a sequence of strictly stationary  $(m-1)$-dependent random variables. Define 
\begin{align*}
U_N^\#=&~\binom{N}{2}^{-1}\sum_{1\leq i<j\leq N}\psi(\zeta_{i,m},\zeta_{j,m}),\\
\psi_1(x)=&~\mathbb{E}\pmb[\psi(\zeta_{1,m},x)\pmb]-\mathbb{E}\psi(\zeta_{1,m},\zeta_{m+1,m})
\\
\text{and }\sigma_\psi^2=&~\mathbb{E}[\psi_1^2(\zeta_{1,m})]+2\sum_{i=2}^{m}\mathbb{E}[\psi_1(\zeta_{1,m})\psi_1(\zeta_{i,m})].
\end{align*}

From \cite{wang_1999} we have the following result.
\begin{lemma}\label{lemma2}
For $\sigma_\psi^2< \infty$ and $\mathbb{E}|\psi_1(\zeta_{1,m})|^3<\infty$,
\begin{align*}
	\lim_{N\to\infty}\sup_{x}\bigg|P\left(\frac{\sqrt{N}(U_N^\#-\mathbb{E}\psi(\zeta_{1,m},\zeta_{m+1,m}))}{2\sigma_\psi}\leq x\right)-\Phi(x)\bigg|= 0,
\end{align*}
where $\Phi$ is distribution function of standard normal random variable.
\end{lemma}
%%%%%%%%%%
Some important identities involved in asymptotic moments of $U_N$ are given in following result.
\begin{lemma}\label{lemma3}
Take $g(x)=\mathbb{E}h(\zeta_{1,m},x)$ then, we have
\begin{align*}
	&\mathbb{E}\pmb(h(\zeta_{1,m},\zeta_{m+1,m})\pmb)=~\mathbb{E}\pmb(g(\zeta_{1,m})\pmb)=\theta,\\
	&\sum_{i=1}^{2m}cov\pmb(h(\zeta_{m,m},\zeta_{4m,m}),h(\zeta_{i,m},\zeta_{6m,m})\pmb)=~
	\sum_{i=1}^{2m}cov\pmb(g\left(\zeta_{m,m}\right),g\left(\zeta_{i,m}\right)\pmb),\\
	&cov\pmb(h(\zeta_{1,m},\zeta_{m+1,m}),\zeta_{1,m}\pmb)=~cov\pmb(g(\zeta_{1,m}),\zeta_{1,m}\pmb),\\
	\text{and }& cov\pmb(h(\zeta_{1,m},\zeta_{m+1,m}),(\zeta_{1,m}-m-1)^2+(\zeta_{m+1,m}-m-1)^2\pmb)\\
	&=~2cov\pmb(g(\zeta_{1,m}),(\zeta_{1,m}-m-1)^2\pmb).
\end{align*}
\end{lemma}
%%%%%%%%%%%%
Using the above results, the asymptotic null distribution of second order U-statistic based on the overlapping $m$-spacings is given by the following result.
\begin{thm}\label{thm1}
Under the null hypothesis, and assumptions (\textbf{A1}), 
\begin{align}
	\sqrt{n}\left(\binom{n-m+1}{2}^{-1}\sum_{1\leq j<k\leq n-m+1}h\left(nT_{j,n}^{(m)},nT_{k,n}^{(m)}\right)-\theta\right)\xrightarrow{d}N\big(0,\sigma^2_h=4(A-B^2)\big),
\end{align}	
as $n\to\infty$, where
\begin{align*}
	A=&~\sum_{j=1}^{2m}cov\pmb(h(\zeta_{m,m},\zeta_{4m,m}),h(\zeta_{j,m},\zeta_{6m,m})\pmb)\text{ and }\\
	B=&~cov\pmb(h(\zeta_{1,m},\zeta_{m+1,m}), \zeta_{1,m}\pmb).
\end{align*}
\end{thm}
\begin{proof}
Let $h^*(x,y)=h(x,y)-\theta$, $x\geq0,~y\geq0$, $N=n-m+1$ and $1\leq M\leq n-3$.
Using Lemma \ref{lemma1}, the characteristic function of $\dfrac{2\sqrt{n}}{MN}\sum\limits_{1\leq j<k\leq M-m}h^*\left(nT_{j,n}^{(m)},nT_{k,n}^{(m)}\right)$ is given by
\begin{align*}
	\Psi_M(t)=\frac{1}{a_n}\int_{-\infty}^{\infty}&\mathbb{E}\bigg[e^{i\frac{2\sqrt{n}}{MN}t \sum\limits_{1\leq j<k\leq M-m+1}h^*\left(\zeta_{j,m},\zeta_{k,m}\right)+iu\sum_{j=1}^{M}(Z_j-1)}\bigg]\\
	&\times\mathbb{E}\bigg[e^{iu\sum_{j=M+1}^{n}(Z_j-1)}\bigg]~du,
\end{align*}
where $a_n={2\pi n^{n-1}e^{-n}/(n-1)!}$. Using Stirling's formula $a_n\approx \sqrt{\frac{2\pi}{n}}$, for large $n$.

Let $n,M\to\infty$ such that $\frac{M}{n}\to\alpha$, where $0<\alpha<1$. Then 
\begin{align*}
	\mathbb{E}\bigg[e^{iu\frac{1}{\sqrt{n}}\sum_{j=M+1}^{n}(Z_j-1)}\bigg]\to e^{-(1-\alpha)u^2/2}.
\end{align*}
By a variant of the Lebesgue DCT,
\begin{align*}
	\Psi_M(t)\to\Psi_\alpha(t)=\frac{1}{\sqrt{2\pi}}\int_{-\infty}^\infty f_{\alpha}(t,u) 
	e^{-(1-\alpha)u^2/2}~du,
\end{align*}
if
\begin{align*}
	f_M(t,u)=\mathbb{E}\bigg[e^{i\frac{2\sqrt{n}}{MN}t \sum\limits_{1\leq j<k\leq M-m+1}h^*\left(\zeta_{j,m},\zeta_{k,m}\right)+iu\frac{1}{\sqrt{n}}\sum_{j=1}^{M}(Z_j-1)}\bigg]\to f_\alpha(t,u).
\end{align*}
Observe that the random quantity
\begin{align*}
	{\frac{2\sqrt{n}}{MN}t \sum\limits_{1\leq j<k\leq M-m+1}h^*\left(\zeta_{j,m},\zeta_{k,m}\right)+u\frac{1}{\sqrt{n}}\sum_{j=1}^{M}(Z_j-1)}
\end{align*}
has the same asymptotic distribution as 
\begin{align*}
	\frac{2\sqrt{n}}{MN} \sum\limits_{1\leq j<k\leq M-m+1}\left(th^*\left(\zeta_{j,m},\zeta_{k,m}\right)+u\frac{\zeta_{j,m}+\zeta_{k,m}-2m}{2m}\right).
\end{align*}
Using Lemma \ref{lemma2}, we get 
\begin{align*}
	&\sqrt{M-m+1}\binom{N-m+1}{2}^{-1}\sum_{1\leq j<k\leq M-m+1}\left(th^*\left(\zeta_{j,m},\zeta_{k,m}\right)+u\frac{\zeta_{j,m}+\zeta_{k,m}-2m}{2m}\right)\\
	&~~\xrightarrow{d}N(0,4\sigma^2(t,u)),\\
	\Rightarrow&~\sqrt{n}\left(\frac{2}{MN} \sum\limits_{1\leq j<k\leq M-m+1}\left(th^*\left(\zeta_{j,m},\zeta_{k,m}\right)+u\frac{\zeta_{j,m}+\zeta_{k,m}-2m}{2m}\right)\right)\xrightarrow{d}N(0,4\alpha\sigma^2(t,u)),
\end{align*}
where
\begin{align*}
	\sigma^2(t,u)=&~V\left(tg(\zeta_{1,m})+u\frac{\zeta_{1,m}-m}{2m}\right)
	+2\sum_{j=2}^{m}Cov\left(tg(\zeta_{1,m})+u\frac{\zeta_{1,m}-m}{2m},tg(\zeta_{j,m})+u\frac{\zeta_{j,m}-m}{2m}\right),\\
	g(x)=&~\mathbb{E}h(\zeta_{1,m},x).
\end{align*}
Hence it follows that
\begin{align*}
	f_M(t,u)\to f_\alpha(t,u)=~e^{-4\alpha\sigma^2(t,u)/2}.
\end{align*}
and, consequently
\begin{align*}
	\Psi_M(t)\to\Psi_\alpha(t)=\frac{1}{\sqrt{2\pi}}\int_{-\infty}^\infty e^{-\frac{1}{2}(4\alpha\sigma^2(t,u)+(1-\alpha)u^2)}~du
\end{align*}
Clearly, $\Psi_\alpha(t)\to 1$ as $\alpha\downarrow0$. By using arguments similar to the proof of Theorem in \cite{holst_1979}, we have
\begin{align*}
	&\mathbb{E}\left(e^{it\sqrt{n}\binom{N}{2}^{-1}\sum\limits_{1\leq j<k\leq N}\left[h\left(nT_{j,n}^{(m)},nT_{k,n}^{(m)}\right)-\theta\right]}\right)\\
	\rightarrow&~ \lim_{\alpha\uparrow1} \Psi_\alpha(t)=\frac{1}{\sqrt{2\pi}}\int_{-\infty}^\infty 
	e^{-2\sigma^2(t,u)}~du.
\end{align*}
Elementary simplification yields
\begin{align*}
	\sigma^2(t,u)=~& t^2\left[V(g(\zeta_{1,m}))+2\sum_{j=2}^{m}Cov(g(\zeta_{1,m}),g(\zeta_{j,m}))\right]\\
	+~&\frac{tu}{m}\left[Cov(g(\zeta_{1,m}),\zeta_{1,m})+ \sum_{j=2}^{m}Cov(g(\zeta_{1,m}),\zeta_{j,m})+ \sum_{j=2}^{m}Cov(g(\zeta_{j,m}),\zeta_{1,m})\right]+~\frac{1}{4}u^2,\\
	=~&t^2A+tuB+\frac{1}{4} u^2.
\end{align*}
Hence,
\begin{align*}
	\frac{1}{\sqrt{2\pi}}\int_{-\infty}^\infty 
	e^{-2\sigma^2(t,u)}~du=~e^{-2t^2(A-B^2)}.
\end{align*}
This concludes the proof.
\end{proof}
%%%%%%%%%%%%%%%%%%%%%%%%%%%%%%%%%%%%%%%%%%%%%%%%%%%%%%%%%%%%%%%%%%%%%%%%%%%%%%%%%%%%%%%%%%%
\section{Asymptotic distribution under a sequence of local alternatives}
In this section, we will discuss the asymptotic distribution of second order U-statistic of the overlapping $m$-spacings, under a popular sequence of local alternatives. For assessing performances of the tests, a popular method is to investigate the performances of tests under a sequence of local alternatives converging to the null distribution at some appropriate rate. We should consider a sequence of local alternatives $\{F_n:n\geq1\}$ that converges to the uniform distribution on $[0,1]$ as $n\to\infty$. For tests based on symmetric sum of overlapping $m$-spacings, an appropriate sequence of alternative (cf. \cite{kuo_1981}) is 
\begin{align}\label{fn_alt}
F_n(x)= x+ \frac{L_n(x)}{\sqrt[4]{n}}, \text{ for } 0\leq x\leq 1,
\end{align}
where $L_n(0)=L_n(1)=0$. Further, assume that $L_n(x)$ is twice differentiable on the unit interval $[0,1]$ and there exists a function $L(x)$ which is twice continuously differentiable with $L(0)=L(1)=0$, such that
\begin{align*}
\sqrt[4]{n}\sup_{\substack{0\leq x \leq 1}}|L_n(x)-L(x)|&= o(1),\\
\sqrt[4]{n}\sup_{\substack{0\leq x \leq 1}}|L_n'(x)-L'(x)|&= o(1),\\
\text{ and } \sqrt[4]{n}\sup_{\substack{0\leq x \leq 1}}|L_n''(x)-L''(x)|&= o(1).
\end{align*}
The asymptotic distribution of symmetric sum function of overlapping $m$-spacings, under the sequence of above mentioned alternatives was obtained by \cite{kuo_1981}. Following Kuo and Rao (1981), we will derive the asymptotic distribution of the proposed second order U-statistic of the overlapping $m$-spacings under the sequence of loacl alternatives (\ref{fn_alt}).

Consider the sequence of real numbers $\{\xi_i\}$, such that
\begin{align*}
\xi_i=\frac{i}{n},\text{ for }i=1,2,\ldots,n.
\end{align*}
Under the sequence of local alternatives (\ref{fn_alt}), $D_{i,n}^{(m)}$ is related to $T_{i,n}^{(m)}$ by following relation
\begin{align}\label{eq_taylor}
nD_{i,n}^{(m)}=&n[F_n^{-1}(U_{i+m-1:n})-F_n^{-1}(U_{i-1:n})]\nonumber\\
=&nT_{i,n}^{(m)}+\left(\frac{-L'(\xi_i)}{\sqrt[4]{n}}+
\frac{[L'(\xi_i)]^2+L(\xi_i)L''(\xi_i)}{\sqrt{n}}\right)nT_{i,n}^{(m)}+o_p(n^{-1/2}),
\end{align}
where $o_p(\cdot)$ is uniform in $i$. This is obtained using the mean value theorem and continuity of $h$  (cf. \cite{rao_1975}).

Assume that $h$ has the first and second order partial derivatives. We denote the first order partial derivatives by $\frac{\partial}{\partial x}h(x,y)=h_x(x,y)$ and $\frac{\partial}{\partial y}h(x,y)=h_y(x,y)$, and the second order partial derivatives by $\frac{\partial^2}{\partial x^2}h(x,y)=h_{xx}(x,y)$, $\frac{\partial^2}{\partial y^2}h(x,y)=h_{yy}(x,y)$ and $\frac{\partial^2}{\partial x\partial y}h(x,y)=h_{xy}(x,y)$. Equation (\ref{eq_taylor}) is used to prove the following result.
%%%%%
\begin{lemma}\label{lemma7}
Under the sequence of local alternatives (\ref{fn_alt}), as $n\to\infty$, 
\begin{align}
	&\sqrt{n}\begin{pmatrix}
		n-m+1\\2\\
	\end{pmatrix}^{-1}\sum_{1\leq i<j\leq n-m+1}\left[h\left(nD_{i,n}^{(m)},nD_{j,n}^{(m)}\right)-h\left(nT_{i,n}^{(m)},nT_{j,n}^{(m)}\right)\right]\nonumber\\
	&\xrightarrow{p}\frac{1}{2} cov[h(\zeta_{1,m},\zeta_{m+1,m}),(\zeta_{1,m}-m-1)^2+(\zeta_{m+1,m}-m-1)^2]
	\int_{0}^{1}(L'(u))^2~du.
\end{align}
\end{lemma}
\begin{proof}
Using (\ref{eq_taylor}) and the two dimensional Taylor expansion of $h(nD_{i,n}^{(m)},nD_{j,n}^{(m)})$ around $(nT_{i,n}^{(m)},nT_{j,n}^{(m)})$, we have, for $1\leq i<j\leq N$
\begin{small}
	\begin{align*}
		&h(nD_{i,n}^{(m)},nD_{j,n}^{(m)})-h(nT_{i,n}^{(m)},nT_{j,n}^{(m)})\\
		=& \left(\frac{-L'(\xi_i)}{\sqrt[4]{n}}+\frac{(L'(\xi_i))^2+L(\xi_i)L''(\xi_i)} {\sqrt{n}}\right)(nT_{i,n}^{(m)})h_x\left(nT_{i,n}^{(m)},nT_{j,n}^{(m)}\right)\\
		&+ \left(\frac{-L'(\xi_j)}{\sqrt[4]{n}}+\frac{(L'(\xi_j))^2+L(\xi_j)L''(\xi_j)} {\sqrt{n}}\right)(nT_{j,n}^{(m)})h_y\left(nT_{i,n}^{(m)},nT_{j,n}^{(m)}\right)\\
		&+ \left(\frac{-L'(\xi_i)}{\sqrt[4]{n}}+\frac{(L'(\xi_i))^2+L(\xi_i)L''(\xi_i)} {\sqrt{n}}\right) \left(\frac{-L'(\xi_j)}{\sqrt[4]{n}}+\frac{(L'(\xi_j))^2+L(\xi_j)L''(\xi_j)} {\sqrt{n}}\right)\\
		&~~~~\times (nT_{i,n}^{(m)})(nT_{j,n}^{(m)}) h_{xy}\left(nT_{i,n}^{(m)},nT_{j,n}^{(m)}\right)\\
		&+ \frac{1}{2} \left(\frac{-L'(\xi_i)}{\sqrt[4]{n}}+\frac{(L'(\xi_i))^2+L(\xi_i)L''(\xi_i)} {\sqrt{n}}\right)^2(nT_{i,n}^{(m)})^2 h_{xx}\left(nT_{i,n}^{(m)},nT_{j,n}^{(m)}\right)\\
		&+ \frac{1}{2} \left(\frac{-L'(\xi_j)}{\sqrt[4]{n}}+\frac{(L'(\xi_j))^2+L(\xi_j)L''(\xi_j)} {\sqrt{n}}\right)^2(nT_{j,n}^{(m)})^2 h_{yy}\left(nT_{i,n}^{(m)},nT_{j,n}^{(m)}\right)+o_p(n^{-1/2}).
	\end{align*}
\end{small}
Therefore,
\begin{small}
	\begin{align}\label{eq26}
		&\sqrt{n}\begin{pmatrix}
			n-m+1\\2\\
		\end{pmatrix}^{-1}\sum_{1\leq i<j\leq n-m+1}\left[h\left(nD_{i,n}^{(m)},nD_{j,n}^{(m)}\right)-h\left(nT_{i,n}^{(m)},nT_{j,n}^{(m)}\right)\right]\nonumber\\
		=&-\sqrt[4]{n}\begin{pmatrix}
			n-m+1\\2\\
		\end{pmatrix}^{-1}\sum_{1\leq i<j\leq n-m+1}L'(\xi_i) (nT_{i,n}^{(m)})h_x\left(nT_{i,n}^{(m)},nT_{j,n}^{(m)}\right)\nonumber\\
		&-\sqrt[4]{n}\begin{pmatrix}
			n-m+1\\2\\
		\end{pmatrix}^{-1}\sum_{1\leq i<j\leq n-m+1}L'(\xi_j) (nT_{j,n}^{(m)})h_y\left(nT_{i,n}^{(m)},nT_{j,n}^{(m)}\right)\nonumber\\
		&+\begin{pmatrix}
			n-m+1\\2\\
		\end{pmatrix}^{-1}\sum_{1\leq i<j\leq n-m+1}[(L'(\xi_i))^2+L(\xi_i)L''(\xi_i)] (nT_{i,n}^{(m)})h_x\left(nT_{i,n}^{(m)},nT_{j,n}^{(m)}\right)\nonumber\\
		&+\begin{pmatrix}
			n-m+1\\2\\
		\end{pmatrix}^{-1}\sum_{1\leq i<j\leq n-m+1}[(L'(\xi_j))^2+L(\xi_j)L''(\xi_j)] (nT_{j,n}^{(m)})h_y\left(nT_{i,n}^{(m)},nT_{j,n}^{(m)}\right)\nonumber\\
		&+\begin{pmatrix}
			n-m+1\\2\\
		\end{pmatrix}^{-1}\sum_{1\leq i<j\leq n-m+1}L'(\xi_i)L'(\xi_j)(nT_{i,n}^{(m)})(nT_{j,n}^{(m)}) h_{xy}\left(nT_{i,n}^{(m)},nT_{j,n}^{(m)}\right)\nonumber\\
		&+\frac{1}{2}\begin{pmatrix}
			n-m+1\\2\\
		\end{pmatrix}^{-1}\sum_{1\leq i<j\leq n-m+1}[L'(\xi_i)]^2 (nT_{i,n}^{(m)})^2h_{xx}\left(nT_{i,n}^{(m)},nT_{j,n}^{(m)}\right)\nonumber\\
		&+\frac{1}{2}\begin{pmatrix}
			n-m+1\\2\\
		\end{pmatrix}^{-1}\sum_{1\leq i<j\leq n-m+1}[L'(\xi_j)]^2 (nT_{j,n}^{(m)})^2h_{yy}\left(nT_{i,n}^{(m)},nT_{j,n}^{(m)}\right)+o_p(1).
	\end{align}
\end{small}
According to the composite trapezoid rule, for any twice continuously differentiable $l$ defined on $[0,1]$ such that $\int_0^1l(u)~du=0$, we have
\begin{align*}
	\lim_{n\to\infty}\sqrt[4]{n}\left[\frac{l''(c)}{12(n+1)^2}-\frac{1}{n}\sum_{k=1}^{n}l(\xi_k)\right]=0,
\end{align*}
where $c\in(0,1)$. 
Thus, the first two terms on the RHS of equation (\ref{eq26}) converge on the probability to zero (cf. \cite{tung_2012a,tung_2012b}).

Using weak law of large numbers and the fact that 
\begin{align*}
	\int_{0}^{1} L(u)L''(u)~du=-\int_{0}^{1} (L'(u))^2~du,
\end{align*}
we get
\begin{align*}
	&\begin{pmatrix}
		n-m+1\\2\\
	\end{pmatrix}^{-1}\sum_{1\leq i<j\leq n-m+1}[(L'(\xi_i))^2+L(\xi_i)L''(\xi_i)] (nT_{i,n}^{(m)})h_x\left(nT_{i,n}^{(m)},nT_{j,n}^{(m)}\right)\\
	&\xrightarrow{p}\mathbb{E}[\zeta_{1,m}h_x(\zeta_{1,m},\zeta_{m+1,m})]\int_{0}^{1} [(L'(u))^2+L(u)L''(u)]~du=0,
\end{align*}
and 
\begin{align*}
	&\begin{pmatrix}
		n-m+1\\2\\
	\end{pmatrix}^{-1}\sum_{1\leq i<j\leq n-m+1}[(L'(\xi_j))^2+L(\xi_j)L''(\xi_j)] (nT_{j,n}^{(m)})h_y\left(nT_{i,n}^{(m)},nT_{j,n}^{(m)}\right)\\
	&\xrightarrow{p}\mathbb{E}[\zeta_{1,m}h_y(\zeta_{1,m},\zeta_{m+1,m})]\int_{0}^{1} [(L'(u))^2+L(u)L''(u)]~du=0.
\end{align*}
Also, observe that
\begin{align*}
	&\begin{pmatrix}
		n-m+1\\2\\
	\end{pmatrix}^{-1}\sum_{1\leq i<j\leq n-m+1}L'(\xi_i)L'(\xi_j) (nT_{i,n}^{(m)})(nT_{j,n}^{(m)}) h_{xy}\left(nT_{i,n}^{(m)},nT_{j,n}^{(m)}\right)\\ &\xrightarrow{p}
	\mathbb{E}[\zeta_{1,m}\zeta_{m+1,m}h_{xy}(\zeta_{1,m},\zeta_{m+1,m})]
	\int_{0}^{1}\int_{0}^{1}L'(u)L'(v)~du~dv=0,
\end{align*}
as
\begin{align*}
	\int_{0}^{1}\int_{0}^{1}L'(u)L'(v)~du~dv=\int_{0}^{1}L'(u)~du~\int_{0}^{1}L'(v)~dv=0.
\end{align*}
Further, observe that
\begin{align*}
	&\frac{1}{2}\begin{pmatrix}
		n-m+1\\2\\
	\end{pmatrix}^{-1}\sum_{1\leq i<j\leq n-m+1}[L'(\xi_i)]^2 (nT_{i,n}^{(m)})^2h_{xx}\left(nT_{i,n}^{(m)},nT_{j,n}^{(m)}\right)\\
	&\xrightarrow{p}\frac{1}{2} \mathbb{E}[\zeta_{1,m}^2h_{xx}(\zeta_{1,m},\zeta_{m+1,m})]
	\int_{0}^{1}(L'(u))^2~du,
\end{align*}

\begin{align*}
	&\frac{1}{2}\begin{pmatrix}
		n-m+1\\2\\
	\end{pmatrix}^{-1}\sum_{1\leq i<j\leq n-m+1}[L'(\xi_j)]^2 (nT_{j,n}^{(m)})^2h_{yy}\left(nT_{i,n}^{(m)},nT_{j,n}^{(m)}\right)\\
	&\xrightarrow{p}\frac{1}{2} \mathbb{E}[\zeta_{m+1,m}^2h_{yy}(\zeta_{1,m},\zeta_{m+1,m})]
	\int_{0}^{1}(L'(v))^2~dv,\text{ as }n\to\infty,
\end{align*}
and
\begin{align*}
	&\mathbb{E}\pmb(\zeta_{1,m}^2h_{xx}(\zeta_{1,m},\zeta_{m+1,m})+\zeta_{m+1,m}^2h_{yy}(\zeta_{1,m},\zeta_{m+1,m})\pmb)\\
	=&cov\pmb(h(\zeta_{1,m},\zeta_{m+1,m}),(\zeta_{1,m}-m-1)^2+(\zeta_{m+1,m}-m-1)^2\pmb).
\end{align*}
This concludes the proof.
\end{proof}
Combining Theorem \ref{thm1} and Lemma \ref{lemma7}, we have the following result.
\begin{thm}\label{thm2}
Under the sequence of local alternatives (\ref{fn_alt}) and assumption (\textbf{A1}), as $n\to\infty$, we have
\begin{align}
	&\sqrt{n}\left(\begin{pmatrix}
		n-m+1\\2\\
	\end{pmatrix}^{-1}\sum_{1\leq i<j\leq n-m+1}h\left(nD_{i,n}^{(m)},nD_{j,n}^{(m)}\right)-\mathbb{E}[h(\zeta_{1,m},\zeta_{m+1,m})]\right)\nonumber\\
	&~~\xrightarrow{d}N(\mu_h,\sigma_h^2),
\end{align}	
where
\begin{align*}
	\mu_h=&\frac{1}{2} cov\pmb(h(\zeta_{1,m},\zeta_{m+1,m}),(\zeta_{1,m}-m-1)^2+(\zeta_{m+1,m}-m-1)^2\pmb)
	\int_{0}^{1}(L'(u))^2~du,\\
	\text{and }&\sigma_h^2 \text{ as in Theorem \ref{thm1}}.
\end{align*}
\end{thm}
\begin{proof}
Observe that
\begin{small}
	\begin{align}\label{eq28}
		&\sqrt{n}\left(\begin{pmatrix}
			n-m+1\\2\\
		\end{pmatrix}^{-1}\sum_{1\leq i<j\leq n-m+1}h\left(nD_{i,n}^{(m)},nD_{j,n}^{(m)}\right)-\mathbb{E}[h(\zeta_{1,m},\zeta_{m+1,m})]\right)\nonumber\\
		=&~\sqrt{n}\left(\begin{pmatrix}
			n-m+1\\2\\
		\end{pmatrix}^{-1}\sum_{1\leq i<j\leq n-m+1}h\left(nT_{i,n}^{(m)},nT_{j,n}^{(m)}\right)-\mathbb{E}[h(\zeta_{1,m},\zeta_{m+1,m})]\right)\nonumber\\
		&+\sqrt{n}\begin{pmatrix}
			n-m+1\\2\\
		\end{pmatrix}^{-1}\sum_{1\leq i<j\leq n-m+1}\left[h\left(nD_{i,n}^{(m)},nD_{j,n}^{(m)}\right)-h\left(nT_{i,n}^{(m)},nT_{j,n}^{(m)}\right)\right].
	\end{align}
\end{small}
Using Lemma \ref{lemma7}, the second term on the RHS of (\ref{eq28}) converges in probability to a constant $\mu_h$. Hence using Theorem \ref{thm1} and Slutsky's Theorem, the proof is concluded.
\end{proof}
%%%%%%%%%%%%%%%%%%%%%%%%%%%%%%%%%%%%%%%%%%%%%%%%%%%%%%%%%%%%%%%%%%%%%%%%%%%%%%%
\section{Asymptotically locally most powerful test}
For the general test statistic $V_{m,n}(g)$ of type  (\ref{one}), asymptotic distribution under the null hypothesis and the sequence of local alternatives (\ref{fn_alt}) is known. With proper centring and scaling of $V_{m,n}(g)$, limiting distribution under the null hypothesis, has zero mean and finite variance (Rao \& Kuo 1981). Suppose that under a sequence of local alternatives the same statistic has asymptotic distribution with mean $\mu_g$ and variance $\sigma_g^2$. The Pitman asymptotic relative efficiency (ARE) of $V_{m,n}(g_1)$ relative to $V_{m,n}(g_2)$ is defined as 
\begin{align*}
ARE(g_1,g_2)=\left(\frac{e^2(g_1)}{e^2(g_2)}\right)^2,\text{ where }e^2(g)=\frac{\mu_g^2}{\sigma_g^2}.
\end{align*}
The quantity $e^2(g)$ is known as the efficacy of the test based on $V_{m,n}(g)$. Asymptotically locally most powerful (ALMP) test is the test with maximum efficacy. To find the ALMP test based on $V_{m,n}(g)$, one need to find $g$, which maximizes
\begin{align*}
e(g)=\frac{\left(\int_{0}^{1}(L'(u))^2du\right)cov\pmb(g(\zeta_{1,m}), (\zeta_{1,m}-m-1)^2\pmb)} {2\left(\sum_{i=1}^{2m}cov\pmb(g\left(\zeta_{m,m}\right),g\left(\zeta_{i,m}\right)\pmb)
	-\left[cov\pmb(g\left(\zeta_{1,m}\right),\zeta_{1,m}\pmb)\right]^2\right)^{1/2}}.
\end{align*}
Kuo and Rao (1981) found that $e(g)$ is maximized for $g(t)=t^2$, and this corresponds to the generalized Greenwood test based on overlapping spacings. 

From Theorem \ref{thm2}, we know that the asymptotic mean and variance of the general test statistic $W_{m,n}(h)$ of the form (\ref{two}) under the sequence of alternatives (\ref{fn_alt}). Under the null hypothesis the asymptotic mean is zero and variance is finite. Thus, the Pitman ARE of test based on $W_{m,n}(h_1)$ with respect to test based on $W_{m,n}(h_2)$, is given by
\begin{align}
ARE(h_1,h_2)=\left(\frac{e^2(h_1)}{e^2(h_2)}\right)^2,
\end{align}
where $e^2(h)=\frac{\mu_h^2}{\sigma_h^2}$ is the efficacy of the test based on $W_{m,n}(h)$. To obtain the ALMP test based on second order U-statistic of the overlapping $m$-spacings, we should find a function $h$, which maximizes the functional
\begin{align}\label{efficacy}
&e(h)=\frac{cov\pmb(h(\zeta_{1,m},\zeta_{m+1,m}),(\zeta_{1,m}-m-1)^2+(\zeta_{m+1,m}-m-1)^2\pmb)
	\int_{0}^{1}(L'(u))^2~du}{4\left(\sum_{i=1}^{2m}cov\pmb(h(\zeta_{m,m},\zeta_{4m,m}),h(\zeta_{i,m},\zeta_{6m,m})\pmb)-\pmb[cov\pmb(h(\zeta_{1,m},\zeta_{m+1,m}), \zeta_{1,m}\pmb)\pmb]^2\right)^{1/2}}.
\end{align}
\begin{lemma}\label{lemma8}
The value of functional $e(h)$ is maximized for the symmetric function $h(x,y)=(x-y)^2$. The maximum value of $e(h)$ is given by
\begin{align}
	\max_{h}~ e(h)=\sqrt{\frac{3m(m+1)} {{2(2m+1)}}}\int_{0}^{1}(L'(u))^2~du.
\end{align}
\end{lemma}
\begin{proof}
Since efficacy is not affected by linear transformations, it is sufficient to find $h$ for which the numerator of (\ref{efficacy}) is maximum. By the Cauchy-Schwarz inequality
\begin{align}
	&cov\pmb(h(\zeta_{1,m},\zeta_{m+1,m}),(\zeta_{1,m}-m-1)^2+(\zeta_{m+1,m}-m-1)^2\pmb)\nonumber\\
	%\int_{0}^{1}(L'(u))^2~du\nonumber\\
	\leq &\sqrt{var\left(h(\zeta_{1,m},\zeta_{m+1,m})\right)} \sqrt{var\left((\zeta_{1,m}-m-1)^2+(\zeta_{m+1,m}-m-1)^2\right)},%\int_{0}^{1}(L'(u))^2~du,
\end{align}
equality is attained if and only if $h(x,y)=a[(x-m-1)^2+(y-m-1)^2]+b$, with $a\in\mathbb{R}\setminus\{0\}$ and $b\in\mathbb{R}$.	

Observe that $cov[\zeta_{1,m}\zeta_{m+1,m},(\zeta_{1,m}-m-1)^2+(\zeta_{m+1,m}-m-1)^2]=0$, and hence $e(h)$ is maximum when $h(x,y)=(x-y)^2$.

Thus,
\begin{align*}
	\max e(h)=&\frac{var\left((\zeta_{1,m}-m-1)^2+(\zeta_{m+1,m}-m-1)^2\right)\int_{0}^{1}(L'(u))^2~du} {4\sqrt{2m(m+1)(2m+1)/3}}\\
	=&\sqrt{\frac{3m(m+1)} {{2(2m+1)}}}\int_{0}^{1}(L'(u))^2~du.
\end{align*}
This concludes the proof.
\end{proof}
Combining Theorem \ref{thm2} and Lemma \ref{lemma8}, we have the following result.
\begin{thm}\label{thm3}
Among the tests based on second order U-statistic of the overlapping $m$-spacings, for the null hypothesis against the sequence of local alternatives (\ref{fn_alt}), the ALMP test at $\alpha$ level of significance, is that which rejects the null hypothesis when
\begin{align}
	\sum_{1\leq i<j\leq n-m+1}h\left(nD_{i,n}^{(m)},nD_{j,n}^{(m)}\right)> C(\alpha),~0<\alpha<1,
\end{align}
where $C(\alpha)$ is the critical value. The asymptotic distribution of this ALMP tests statistic under the sequence of local alternatives (\ref{fn_alt}) is given by
\begin{align}
	&\sqrt{n}\left(\begin{pmatrix}
		n-m+1\\2\\
	\end{pmatrix}^{-1}\sum_{1\leq i<j\leq n-m+1}\left(nD_{i,n}^{(m)}-nD_{j,n}^{(m)}\right)^2-2m\right)\nonumber\\
	&\xrightarrow{d}N\left(2m(m+1)\int_{0}^{1}(L'(u))^2~du,\frac{8m(m+1)(2m+1)}{3}\right).
\end{align}	
Under the null hypothesis, the asymptotic distribution can be obtained by putting $L'(u)=0$ in the above.
\end{thm} 

Thus under the sequence of local alternatives (\ref{fn_alt}), the test based on Gini's mean squared difference of overlapping $m$-spacings is ALMP among the tests based on second order U statistic of overlapping $m$-spacings. The Gini's mean squared difference test statistic is given by
\begin{align}
G_{m,n}(2)=\frac{2}{N(N-1)}\sum_{1\leq i<j\leq n-m+1}\left|nD_{i,n}^{(m)}-nD_{j,n}^{(m)}\right|^2.
\end{align}
Also this test is ALMP among tests based on generalized Gini's mean difference statistics of overlapping $m$-spacings. Using Lemma \ref{lemma8}, efficacy of the Gini's mean squared difference test is given by
\begin{align}\label{eff_Gini's}
e^2(h)=\frac{3m(m+1)} {{2(2m+1)}}\left(\int_{0}^{1}(L'(u))^2~du\right)^2.
\end{align}
The Gini's mean squared difference test based on simple (or disjoint $m$-) spacings is ALMP among second order U-statistics tests  based on simple (or disjoint $m$-) spacings. Further, the Gini's mean squared difference test based on simple (or disjoint $m$-) spacings has the same efficacy as the Greenwood test based on simple (or disjoint $m$-) spacings (\cite{tung_2012a,tung_2012b}). They also found that as $m$ increases, the efficacy of the test increases. Theorem \ref{thm3} conveys that among second order U-statistic tests based on overlapping $m$-spacings, Gini's mean squared difference test is ALMP. Moreover, as expected, this tests efficacy is same as the Greenwood test based on overlapping $m$-spacings. 

For different values of $m$, \cite{rao_1984} studied the efficacies of Greenwood tests based on disjoint and overlapping $m$-spacings. They found that the test based on overlapping $m$-spacings is superior. It follows the Gini's mean squared difference tests based on overlapping $m$-spacings is superior to that based on disjoint $m$-spacings in terms of the Pitman ARE. A table of efficacies can be found in Rao \& Kuo (1984). \cite{misra_2001} compared Greenwood and Kullback-Leibler tests based on overlapping $m$-spacings. They found that the Greenwood test is superior in terms of the Pitman ARE. Thus, the Gini's mean squared difference tests based on overlapping $m$-spacings is superior to the Kullback-Leibler tests based on overlapping $m$-spacings.\\~\\
%%%%%%%%%%%%%
\textbf{Remark:}
GoF tests based on spacings is a class of tests which can be directly applied to real as well as circular data. \cite{rao_2004} discussed Gini's mean difference test based on simple spacings for uniformity of circular data. Further \cite{tung2013} extended the study to generalized Gini's mean difference tests. They showed that to  Gini's mean square difference test is ALMP in terms of the Pitman ARE for alternatives of type (\ref{fn_alt}). This is equivalent to the Greenwood test. Using the results established in previous Sections 2-4, these results hold for the case of overlapping spacings as well. That is, in terms of the Pitman ARE, the  Gini's mean square difference test is ALMP among among generalized Gini's mean difference tests based on overlapping $m$-spacings, for the alternatives of type (\ref{fn_alt}).

\section{Simulation study}
In this section, we report finite sample performance of some generalized Gini's mean difference tests based overlapping $m$-spacings. Though the Greenwood test is ALMP among the tests based on simple spacings, \cite{rao_2004} found that Gini's mean difference test is superior to the Greenwood test based on simple spacings, for finite sample cases.

For the simulation study, we consider generalized Gini's mean difference tests based on disjoint and overlapping spacings corresponding to $r=1,~1.5$ and $2$. We take $U(0,1)$ as the null distribution and level of significance to be $0.05$. $Beta(0.5,0.5)$, $Beta(3,3)$ and $Beta(1,3)$ alternatives are considered. $Beta(0.5,0.5)$ has heavier tail, $Beta(3,3)$ has lighter tail as compared to $U(0,1)$ and $Beta(1,3)$ is skewed.
\begin{table}[h]
\centering
\caption{Empirical powers for the alternative $Beta(0.5,0.5)$ and $n=50$.}
\begin{tabular}{lllllll}
	\hline
	& \multicolumn{3}{l}{Disjoint} & \multicolumn{3}{l}{Overlapping} \\ \cline{2-7} 
	$m$ & $r=1$   & $r=1.5$  & $r=2$   & $r=1$    & $r=1.5$   & $r=2$    \\ \hline
	1   & 0.6148  & 0.5086   & 0.4277  & 0.6148   & 0.5086    & 0.4277   \\
	2   & 0.6701  & 0.6075   & 0.5192  & 0.7237   & 0.6671    & 0.5941   \\
	4   & 0.7073  & 0.6626   & 0.6162  & 0.7797   & 0.7313    & 0.7093   \\
	5   & 0.6349  & 0.6158   & 0.5575  & 0.7711   & 0.7483    & 0.7274   \\
	10  & 0.5652  & 0.5591   & 0.5483  & 0.7129   & 0.7013    & 0.6909   \\ \hline
\end{tabular}
\end{table}
%%%%%%%%%%%%%%%%%%%%%%%%%%%%%%
\begin{table}[h]
\centering
\caption{Empirical powers for the alternative $Beta(3,3)$ and $n=50$.}
\begin{tabular}{lllllll}
	\hline
	& \multicolumn{3}{l}{Disjoint} & \multicolumn{3}{l}{Overlapping} \\ \cline{2-7} 
	$m$ & $r=1$   & $r=1.5$  & $r=2$   & $r=1$    & $r=1.5$   & $r=2$    \\ \hline
	1   & 0.7631  & 0.8553   & 0.8467  & 0.7631   & 0.8553    & 0.8467   \\
	2   & 0.2773  & 0.4677   & 0.5658  & 0.2906   & 0.5555    & 0.6759   \\
	4   & 0.2982  & 0.5016   & 0.5931  & 0.1434   & 0.1701    & 0.3229   \\
	5   & 0.0052  & 0.0078   & 0.0147  & 0.0792   & 0.1165    & 0.19     \\
	10  & 0.0115  & 0.0172   & 0.0159  & 0.0247   & 0.0368    & 0.0591   \\ \hline
\end{tabular}
\end{table}
%%%%%%%%%%%%%%%%%%%%%%%%%%
\begin{table}[h]
\centering
\caption{Empirical powers for the alternative $Beta(1,3)$ and $n=50$.}
\begin{tabular}{lllllll}
	\hline
	& \multicolumn{3}{l}{Disjoint} & \multicolumn{3}{l}{Overlapping} \\ \cline{2-7} 
	$m$ & $r=1$   & $r=1.5$  & $r=2$   & $r=1$    & $r=1.5$   & $r=2$    \\ \hline
	1   & 0.9946  & 0.9978   & 0.9977  & 0.9946   & 0.9978    & 0.9977   \\
	2   & 0.9991  & 0.9995   & 0.9999  & 0.8869   & 0.9963    & 0.9984   \\
	4   & 0.9999  & 0.9999   & 0.9999  & 0.6724   & 0.9578    & 0.9887   \\
	5   & 0.2599  & 0.3245   & 0.3426  & 0.6363   & 0.9143    & 0.9801   \\
	10  & 0.4539  & 0.4572   & 0.4935  & 0.4827   & 0.7253    & 0.9116   \\ \hline
\end{tabular}
\end{table}
%\FloatBarrier
Gini's mean difference test based on overlapping spacings is superior to all the other tests considered for heavy tailed alternative ($Beta(0.5,0.5)$). For other two alternatives the Greenwood tests is superior. The optimal value of $m$ increases as $n$ increases, but for moderate sample sizes optimal value of $m$ is $4$ or $5$. 
\FloatBarrier
\section{Conclusion}
In this article, we have derived asymptotic distribution of second order U-statistic based on overlapping $m$-spacings, under the null hypothesis and a sequence of local alternatives. It has been found that asymptotically locally most powerful test among the class of tests based on second order U-statistics of overlapping $m$-spacings, is the Gini's mean square difference test, which is algebraically equivalent to the Greenwood test based on overlapping $m$-spacings. Simulation study reveals that for heavy tailed alternatives Gini's mean difference test is superior to the Greenwood test based on overlapping spacings. For light tailed and skewed alternatives, the Greenwood test is superior. Optimal choice of $m$ is definitely an interesting question to be answered. For moderate sample sizes optimal value of $m$ is $4$ or $5$. 
%\newpage

%\small
\bibliography{mybibfile}

\end{document}